\documentclass[12pt]{amsart}
\usepackage{etex} 
\usepackage[active]{srcltx}
\usepackage{calc,amssymb,amsthm,amsmath,amscd, eucal,ulem}
\usepackage{alltt}
\RequirePackage[dvipsnames,usenames]{color}

\input{mabliautoref.sty}
\input{xy}
\xyoption{all}
\input{kmacros3.sty}
\usepackage{tikz}
\usepackage{amsfonts, mathrsfs}
\usepackage{calligra}
\usepackage{stmaryrd} 
\usepackage{dcpic, pictexwd}
\usepackage[left=1in,top=1in,right=1in,bottom=1in]{geometry}

\numberwithin{equation}{theorem}

\newcommand{\piet}{\pi_1^{\textnormal{\'{e}t}}} 

\usepackage{setspace}
\usepackage{hyperref}
\usepackage{enumerate}
\usepackage{graphicx}
\usepackage[all,cmtip]{xy}

\usepackage{verbatim}

\theoremstyle{theorem}

\def\todo#1{\textcolor{Mahogany}%
{\footnotesize\newline{\color{Mahogany}\fbox{\parbox{\textwidth-15pt}{\textbf{todo:
} #1}}}\newline}}

\renewcommand{\O}{\mathscr O}

\begin{document}
\title{\'Etale fundamental groups of strongly $F$-regular schemes}
\author[B.~Bhatt]{Bhargav Bhatt}
\author[J.~Carvajal-Rojas]{Javier Carvajal-Rojas}
\author[P.~Graf]{Patrick Graf}
\author[K.~Schwede]{Karl Schwede}
\author[K.~Tucker]{Kevin Tucker}
\address{Department of Mathematics\\ University of Michigan\\Ann Arbor\\ MI 48109}
\email{\href{mailto:bhargav.bhatt@gmail.com}{bhargav.bhatt@gmail.com}}
\address{Department of Mathematics\\ University of Utah\\ Salt Lake City\\ UT 84112 \newline\indent
Escuela de Matem\'atica\\ Universidad de Costa Rica\\ San Jos\'e 11501\\ Costa Rica}
\email{\href{mailto:carvajal@math.utah.edu}{carvajal@math.utah.edu}}
\address{Lehrstuhl f\"ur Mathematik I, Universit\"at Bay\-reuth, 95440 Bayreuth, Germany}
\email{\href{mailto:patrick.graf@uni-bayreuth.de}{patrick.graf@uni-bayreuth.de}}
\address{Department of Mathematics\\ University of Utah\\ Salt Lake City\\ UT 84112}
\email{\href{mailto:schwede@math.utah.edu}{schwede@math.utah.edu}}
\address{Department of Mathematics\\ University of Illinois at Chicago\\Chicago\\ IL 60607}
\email{\href{mailto:kftucker@uic.edu}{kftucker@uic.edu}}

\thanks{Bhatt was supported in part by the NSF Grant DMS \#1501461 and a Packard fellowship}
\thanks{Carvajal-Rojas was supported in part by the NSF FRG Grant DMS \#1265261/1501115}
\thanks{Graf was supported in part by the DFG grant ``Zur Positivit\"at in der komplexen
Geometrie''}
\thanks{Schwede was supported in part by the NSF FRG Grant DMS \#1265261/1501115 and NSF CAREER Grant DMS \#1252860/1501102}
\thanks{Tucker was supported in part by NSF Grant DMS \#1602070 and a fellowship from the Sloan foundation.}



\begin{abstract}
We prove that a strongly $F$-regular scheme $X$ admits a finite, generically Galois, and \'etale-in-codimension-one  cover $\tld X \to X$ such that the \'etale fundamental groups of $\tld X$ and $\tld X_{\reg}$ agree.
Equivalently, every finite \'etale cover of $\tld X_{\reg}$ extends to a finite \'etale cover of $\tld X$.
This is analogous to a result for complex klt varieties by Greb, Kebekus and Peternell.
\end{abstract}

\maketitle

\section{Introduction}

In \cite{XuFinitenessOfFundGroups}, Chenyang Xu proved that the algebraic local fundamental group (the profinite completion of the topological fundamental group of the link) of a complex klt singularity is finite.
This result was then used by Greb--Kebekus--Peternell \cite{GrebKebekusPeternellEtaleFundamental} to show that for a complex quasi-projective variety $X$ with klt singularities, any sequence of finite quasi-\'etale (\ie~\'etale in codimension one) and generically Galois covers eventually becomes \etale.
In particular, for any quasi-projective klt variety $X/\C$ there exists a finite quasi-\'etale generically Galois cover $\rho\colon \tld X \to X$ such that any further quasi-\'etale cover $Y \to \tld X$ is actually \'etale.

Inspired by the relation between klt singularities and strongly $F$-regular singularities in characteristic $p > 0$ \cite{HaraWatanabeFRegFPure}, a subset of the authors showed that if $(R, \fram)$ is a strictly Henselian strongly $F$-regular local ring and $U \subseteq \Spec R$ is the regular locus, then $\pi_1^{\et}(U)$ is finite \cite{CarvajalSchwedeTuckerEtaleFundFsignature}, a variant of Xu's result in characteristic $p > 0$.  In this paper we use this result, a theorem of Gabber \cite{TravauxdeGabber}, and recent work on the non-local behavior of $F$-signature \cite{deStefaniPolstraYao-Globalizing}, to prove a variant of the result of Greb--Kebekus--Peternell in characteristic $p > 0$.

Philosophically, both these results are studying the obstructions of extending finite \'etale covers of the regular locus $X_{\reg}$ of a scheme $X$ to the whole scheme.  The expected general answer is that for more severe singularities on $X$, there are more obstructions.
Analogous to \cite{GrebKebekusPeternellEtaleFundamental}, we show that schemes with strongly $F$-regular singularities are mild in this sense.

\begin{mainthm*}[\textnormal{\autoref{thm.QuasiEtaleStabilize}, \autoref{cor.MaximalQuasiEtale}}]
Suppose $X$ is an $F$-finite Noetherian integral strongly $F$-regular scheme.
Suppose we are given a sequence of finite surjective quasi-\'etale (\'etale in codimension one) morphisms of normal integral schemes
\[ X = X_0 \xleftarrow{\;\;\gamma_1\;\;} X_1 \xleftarrow{\;\;\gamma_2\;\;} X_2 \xleftarrow{\;\;\gamma_3\;\;} \cdots \]
such that each $X_i/X$ is generically Galois.
Then all but finitely many of the $\gamma_i$ are \'etale.

In particular, there exists a finite quasi-\'etale generically Galois cover $\tld X \to X$ so that any further finite quasi-\'etale cover $Y \to \tld X$ is actually \'etale.
Equivalently, the map $\piet(\tld X_{\reg}) \to \piet(\tld X)$ induced by the inclusion of the regular locus is an isomorphism.
\end{mainthm*}

Note that we do not require any quasi-projectivity hypothesis.
We actually obtain a slightly stronger version, just as in \cite{GrebKebekusPeternellEtaleFundamental}.
As a corollary, we also obtain a variant of \cite[Theorem 1.10]{GrebKebekusPeternellEtaleFundamental}, see \autoref{cor.SimultaneousIndex1Cover}, which says in particular that there exists an integer $N > 0$ such that for every $\bQ$-Cartier divisor $D$ on $X$ with index not divisible by $p$, $N \cdot D$ is in fact Cartier.
Finally, we observe that the map $\tld X \to X$ from our Main~Theorem above is tame wherever it is not \'etale, see \autoref{cor.MaximalTameCover}.

\subsection*{Acknowledgements}

The authors began working on this project while visiting CIRM Luminy in September 2016.

\section{Notation and conventions}

\begin{convention}
All rings and schemes considered will be Noetherian.  We will typically also be working with normal connected and hence integral schemes.
\end{convention}



\begin{notation}
Let $\rho\colon Y \to X$ be a morphism of normal schemes.
\begin{enumerate}
\item
We say that $\rho$ is \emph{quasi-\'etale} if it is quasi-finite and there exists a closed subset $Z \subseteq Y$ of codimension $\ge 2$ such that $\rho|_{Y \setminus Z}\colon Y \setminus Z \to X$ is \'etale.
Notice that if $\rho$ is further finite, then $Y$ coincides with the integral closure of $X$ in $K(Y)$.
\item
\label{not.b.branchlocus}
Assume that $\rho$ is finite and surjective. Let $W \subseteq X$ be the largest Zariski-open subset such that $\rho^{-1}(W) \to W$ is \'etale.
The \emph{branch locus} of $\rho$ is defined to be $X \setminus W$.
It is worth to remark that by Purity of the Branch Locus, if $\rho$ is quasi-\'etale then its branch locus is contained in the non-regular locus of $X$.
\item
If $X$ and $Y$ are connected, we say that $\rho$ finite and surjective is \emph{generically Galois} if the fraction fields extension $K(Y)/K(X)$ is Galois.
In this case, it follows that $\rho^{-1}(W) \to W$ is Galois, where $W$ is as in \autoref{not.b.branchlocus}.
\item
By a \emph{geometric point} $\bar x \in X$ of a scheme $X$ we mean a morphism $\Spec L \to X$, where $L$ is a separably closed field.
Note that $\bar x \to X$ factors through $\Spec k(x) \to X$, where $x \in X$ is the scheme-theoretic image point.
\item
Suppose that $\bar s, \bar t \in X$ are geometric points of a scheme $X$.
A \emph{specialization of geometric points $\bar s \rightsquigarrow \bar t$} is a factorization
\[
\bar s \to \Spec \O_{X, \overline{t}}^{\sh} \to X
\]
where the composition is the map defining the geometric point $\bar s \to X$.
Note that then we get a map $\Spec \O_{X,\overline{s}}^{\sh} \to \Spec \O_{X,\overline{t}}^{\sh}$.
In particular, the image of $\overline{t}$ as a scheme theoretic point on $X$ is contained in the Zariski closure of the image of $\overline{s}$.
Conversely, if $\bar s \in X$ is a geometric point and $t \in \overline{ \{ s \} }$, then we can find a geometric point $\bar t \in X$ with scheme-theoretic image point $t$ such that $\bar s$ specializes to $\bar t$.
\end{enumerate}
\end{notation}

\section{Gabber's constructibility and applications to the fundamental group}

Recall the following theorem of Gabber.

\begin{theorem}\textnormal{\cite[Expos\'e XXI, Th\'eor\`eme 1.1, 1.3]{TravauxdeGabber}}
\label{thm.GabberConstructibility}
Suppose that $X$ is a normal excellent scheme and $Z \subseteq X$ a closed subscheme of codimension $\geq 2$.
Let $j\colon U = X \setminus Z \to X$ be the inclusion.
Then for all finite groups $G$, $R^i j_* (G_U)$ is constructible for $i = 0, 1$.
Here $G_U$ denotes the constant sheaf $G$ on $U$ equipped with the \'etale topology.
\end{theorem}

We use this to obtain:

\begin{proposition} \label{prop.GabberStratification}
Let $i\colon Z \hookrightarrow X$ be a closed subscheme of codimension $\ge 2$ in a normal excellent scheme.
Let $j\colon U = X \setminus Z \hookrightarrow X$ be the complement.
Suppose that
\[ \big| \pi_1^{\et}(\Spec \O_{X,x}^{\sh} \setminus {Z}^{\sh}) \big| \]
is bounded for all geometric points $x \in X$ (here $Z^{\sh} = Z_{X,x}^{\sh}$ is the inverse image of $Z$).
Then there exists a finite stratification $\{ Z_i \}_{i \in I}$ of $X$ into locally closed subschemes $Z_i$ with the following property:
\begin{itemize}
\item[] Fix a specialization $s \rightsquigarrow t$ of geometric points of $Z_i$. 
Consider the map
    \begin{equation}
    \label{eq.MapOfHenselizations}
     \Spec(\O^{\sh}_{X,s}) \longrightarrow \Spec(\O^{\sh}_{X,t}).
    \end{equation}
    Remove the inverse images of $Z$ to obtain
    \begin{equation}
    \label{eq.TubularMapOfHenselizations}
    T_s \xrightarrow{\;\;\alpha\;\;} T_t.
    \end{equation}
    Then applying $\piet(-)$ to \autoref{eq.TubularMapOfHenselizations} yields an isomorphism.
\end{itemize}
\end{proposition}

\begin{proof}
Note there are only finitely many finite groups $G_x = \pi_1^{\et}(\Spec \O_{X,x}^{\sh} \setminus Z^{\sh})$ for geometric points $x \in X$, since there are only finitely many groups of size less than a given number.
Letting $G_x$ also denote the constant sheaf on $U_{\et}$, since $R^1 j_* G_x$ is constructible by \autoref{thm.GabberConstructibility}, we may choose a stratification $\{ Z_i \}$ of locally closed irreducible subschemes of $X$ so that $R^1 j_* G_x$ is locally constant on each $Z_i$ for all geometric points $x$. Note $R^1 j_* G_x$ is zero over $U$ so the induced stratification has the same data as a stratification of $Z$.

Using our stratification, we know that for any geometric point $x \in X$, the stalks of $R^1 j_* G_x$ at $s$ and $t$ are the same.
These stalks however are $H^1(T_s, G_x)$ and $H^1(T_t, G_x)$, respectively \cite[Chapter III, Theorem 1.15]{MilneEtaleCohomology}.
In other words, the functor of isomorphism classes of $G_x$-torsors over \autoref{eq.TubularMapOfHenselizations} produces isomorphisms.
But we have functorial bijections
\[ H^1(T_s, G_x) \cong \factor \Hom_{\cont} \big( \pi_1^{\et}(T_s), G_x \big). G_x. \]
and
\[ H^1(T_t, G_x) \cong \factor \Hom_{\cont} \big( \pi_1^{\et}(T_t), G_x \big). G_x., \]
where $G_x$ acts on $\Hom_{\cont}(-, G_x)$ by conjugation.
For more discussion see for instance \cite[Section 11.5]{GortzWedhornAlgebraicGeoemtryI}, \cite[Example 11.3]{MilneLecturesOnEtaleCohomology} and \cite[Chapter I, Remark 5.4 and Chapter III, Corollary 4.7, Remark 4.8]{MilneEtaleCohomology}.
In particular, the map
\begin{equation} \label{eq.alpha}
\factor \Hom \big( G_t \cong \piet(T_t), G_x \big). G_x. \longrightarrow \factor \Hom \big( G_s = \piet(T_s), G_x \big). G_x.
\end{equation}
induced by $\piet(\alpha)$ is a bijection for any geometric point $x \in X$.

Recall that we want to prove that $\piet(\alpha)$ is an isomorphism.
We apply~\autoref{eq.alpha} with $x = s$.
Considering the identity $\pi_1^{\et}(T_s) \to G_s$ as an element of the right-hand side, observe that there exists a homomorphism $\kappa\colon G_t \to G_s$, unique up to conjugacy, so that
\begin{equation} \label{eq.kappa}
G_s \xrightarrow{\piet(\alpha)} G_t \xrightarrow{\;\;\kappa\;\;} G_s
\end{equation}
is equal to a conjugate of the identity, that is, an inner automorphism.
In particular, $\kappa$ is surjective and $\pi_1^{\et}(\alpha)$ is injective.
We now apply $\Hom(-, G_t) / G_t$ to~\autoref{eq.kappa}. This gives us
\[ \factor \Hom(G_s, G_t). G_t. \longrightarrow \factor \Hom(G_t, G_t). G_t. \longrightarrow \factor \Hom(G_s, G_t). G_t.. \]
The second map, which is induced from $\piet(\alpha)$, is a bijection by~\autoref{eq.alpha} applied with $x = t$.
The composition is also a bijection, by construction.
Hence the first map, which is induced by $\kappa$, is bijective too.
By the same argument as above, this implies that there exists a homomorphism $\lambda\colon G_s \to G_t$ such that the composition $G_t \xrightarrow{\kappa} G_s \xrightarrow{\lambda} G_t$ is an inner automorphism.
In particular, $\kappa$ is injective and hence an isomorphism.
By~\autoref{eq.kappa}, it follows that $\piet(\alpha)$ is likewise an isomorphism, which was to be shown.
\end{proof}

\begin{remark} \label{rem.StratificationBehavesWellUnderPullback}
The crucial property of the stratification $\{ Z_i \}$ in \autoref{prop.GabberStratification} is preserved by finite quasi-\'etale covers:

Suppose that $\rho\colon X' \to X$ is a separated quasi-finite cover with branch locus contained in $Z$.
By Zariski's Main Theorem, we may assume that $\rho$ is finite since the statement is obvious for open inclusions.
Consider a specialization $s' \rightsquigarrow t'$ of geometric points of $\rho^{-1}(Z_i)$ mapping to a specialization of geometric points of $Z_i \subseteq X$, $s \rightsquigarrow t$, under $\rho$.
The diagram
\begin{equation} \label{fibre product}
\xymatrix{
\Spec \O_{X',s'}^{\sh} \ar[r] \ar[d]_-{\rho_s} & \Spec \O_{X',t'}^{\sh} \ar[d]^-{\rho_t} \\
\Spec \O_{X,s}^{\sh} \ar[r] & \Spec \O_{X,t}^{\sh}
}
\end{equation}
is a fibre product diagram up to taking a connected component (in the end it will follow that there can be at most one connected component from our assumptions).

Denote $T_{t'} := \Spec \O_{X', t'}^{\sh} \setminus \text{inverse image of $\rho^{-1}(Z)$}$ and likewise with $T_{s'}$.
Then removing the preimages of $Z$ from \eqref{fibre product} yields a fibre product diagram (up to taking a connected component, for now)
\[ \xymatrix{
T_{s'} \ar[rr]^-{\alpha'} \ar[d]_-{\rho_s} & & T_{t'} \ar[d]^-{\rho_t} \\
T_s \ar[rr]^-\alpha & & T_t.
} \]
The vertical maps are \'etale since the branch locus of $\rho$ is contained in $Z$, so the induced maps on $\pi_1^{\et}$ are injective.
Since $\piet(\alpha)$ is an isomorphism, we have that $\piet(T_{s'})$ and $\piet(T_{t'})$ define the same subgroups in $\piet(T_s) \cong \piet(T_t)$ (up to conjugation).
In particular, $\piet(\alpha')$ is also an isomorphism and so $\{ \rho^{-1}(Z_i) \}_{i \in I}$ stratifies $X'$ in the sense of \autoref{prop.GabberStratification}.
%
\end{remark}


\begin{proposition} \label{prop.StrataAndCovers}
Suppose that $Z, U \subseteq X$ and let $\{ Z_i \}_{i \in I}$ be such a stratification as in \autoref{prop.GabberStratification} made up of connected $Z_i$.
Let $\rho\colon Y \to X$ be a finite surjective map whose branch locus is contained in $Z$ (in particular $\rho$ is quasi-\'etale), and let $W \subseteq X$ be the maximal Zariski-open subset over which $\rho$ is \'etale.
Then, for every $i \in I$, either $W \cap Z_i = \emptyset$ or $Z_i \subseteq W$.
Equivalently, either the branch locus of $\rho$ contains $Z_i$ or it is disjoint from $Z_i$.
\end{proposition}

\begin{proof}
We have the following criterion for geometric points of $X$ to belong to $W$.

\begin{claim} \label{clm.EtaleIffLocallyTrivial}
A geometric point $t \in X$ is contained in $W$ ($\rho$ is \etale{} over $t$) if and only if the pullback of $\rho\colon Y \to X$ to $T_t := \Spec \O_{X, t}^{\sh}\setminus Z^{\sh}$ is trivial (a finite disjoint union of copies of $T_t$).
\end{claim}

\begin{proof}[Proof of \autoref{clm.EtaleIffLocallyTrivial}]
Since the condition of being \'etale is Zariski-local, we see that $t \in W$ if and only if $\rho_{t}\colon Y_t = Y \times_{X} \Spec \O_{X,t} \to \Spec \O_{X,t}$ is \'etale (here we identify $t$ with its image at $X$ to avoid cumbersome phrasing and notation). However, by faithfully flat descent~\cite[Tag 02YJ, Lemma 34.20.29]{stacks-project} this morphism is \'etale if and only if its pullback to $\Spec \O_{X,t}^{\sh}$, say $\rho_t^{\sh}\colon Y_t^{\sh} \to \Spec \O_{X,t}^{\sh}$, is \'etale. But the target being the spectrum of a strictly henselian local ring, the latter condition implies that the finite \etale{} cover $\rho_t^{\sh}$ is trivial (\ie a finite disjoint union of copies of $\Spec \O_{X,t}^{\sh}$).
In particular, if $t \in W$ then the pullback $\rho^{-1}(T_t) \to T_t$ is trivial.

Conversely if $\rho^{-1}(T_t) \to T_t$ is trivial, then so is $\rho_t^{\sh}$ since\footnote{Note finite maps between normal schemes are determined outside a set of codimension 2 \cite{HartshorneGeneralizedDivisorsOnGorensteinSchemes}.} all the schemes are normal and the complement of $T_t \subseteq \Spec \O_{X,t}^{\sh}$ has codimension $\geq 2$.
As above this implies $t \in W$.
\end{proof}

Now let $t \in Z_i$ be a geometric point and let $s \in Z_i$ be a generic geometric point with a specialization $s \rightsquigarrow t$.
It is sufficient to show that $t \in W$ if and only if $s \in W$.
This follows by recalling that we have a canonical map $T_s \to T_t$ which induces an isomorphism on the level of fundamental groups, and moreover a commutative diagram
\[
\xymatrix@R=12pt{
T_t \ar[rd] \\
T_s \ar[u]\ar[r] & W \subseteq X. \\
}
\]
It follows that the images of $\piet(T_t)$ and of $\piet(T_s)$ coincide in $\piet(W)$ up to conjugation.
By~\cite[Chapter~I, Theorem~5.3]{MilneEtaleCohomology}, the finite \'etale cover $\rho^{-1}(W) \to W$ corresponds to a continuous $\piet(W)$-action on the finite set $Q := F(\rho^{-1}(W))$, where $F$ is the fibre functor.

Now assume $s \in W$, \ie~$\rho$ is \'etale over $s$.
By~\autoref{clm.EtaleIffLocallyTrivial},
the induced action of $\piet(T_s)$ on $Q$ is trivial.
It follows that the induced action of $\piet(T_t)$ on $Q$ is also trivial.
By~\cite[Chapter I, Theorem 5.3]{MilneEtaleCohomology} again, $\rho^{-1}(T_t) \to T_t$, or more precisely $Y \times_X T_t \to T_t$, is trivial, \ie~the disjoint union of $\deg(\rho)$ copies of $T_t$.
Hence by \autoref{clm.EtaleIffLocallyTrivial} again, we get $t \in W$.
Conversely, $t \in W$ implies $s \in W$ by running the argument backwards.
\end{proof}

\section{Maximal quasi-\'etale covers}

As mentioned, our result is somewhat more general than the one in the introduction.

\begin{theorem}\textnormal{(\cf~\cite[Theorem 2.1]{GrebKebekusPeternellEtaleFundamental})} \label{thm.QuasiEtaleStabilize}
Suppose $X$ is an $F$-finite Noetherian integral strongly $F$-regular scheme.
Suppose that we have a commutative diagram of separated quasi-finite maps between normal $F$-finite Noetherian integral schemes
\[
\xymatrix{
    &              Y_1 \ar@{->>}[d]_{\eta_1}  & \ar[l]_{\gamma_1} Y_2 \ar@{->>}[d]_{\eta_2} & \ar[l]_{\gamma_2} Y_3 \ar@{->>}[d]_{\eta_3} & \ar[l]_{\gamma_3} \ldots\\
X   & \ar@{_{(}->}[l]^{j_0}   X_1                  & \ar@{_{(}->}[l]^{j_1}      X_2                 & \ar@{_{(}->}[l]^{j_2}      X_3                 & \ar@{_{(}->}[l]^{j_3} \ldots
}
\]
such that the following conditions hold.
\begin{itemize}
\item[(i)] The maps $j_i$ are inclusions of open sets.
\item[(ii)] The maps $\eta_i$ are finite, quasi-\'etale and generically Galois.
\end{itemize}
Then all but finitely many of the $\gamma_j$ are \'etale.
\end{theorem}

To recover the statement in the introduction, simply set all the $X_i = X$.

\begin{proof}
By \cite[Theorem B]{deStefaniPolstraYao-Globalizing} (see also the semi-continuity of Hilbert-Kunz multiplicity and $F$-signature \cite{SmirnovUpperSemicontinuity,PolstraSemicontinuityFsignature,TuckerPolstraUniformApproach}) and the fact that $X$ is quasi-compact, we know there exists a uniform lower bound $\delta > 0$ on $s(\O_{X,x} =: R_x)$ for each scheme-theoretic point $x$ of $X$.
Let $R_x^{\sh}$ denote the strict henselization of the local ring $R_x$.
Since $R_x \subseteq R_x^{\sh}$ is flat with regular closed fiber, we know that $s(R_x^{\sh}) = s(R_x) \ge \delta$ by \cite[Theorem 5.6]{YaoObservationsAboutTheFSignature}.
Also $R_x^{\sh}$ is $F$-finite since quite generally, the strict henselization of a normal $F$-finite local ring is again $F$-finite\footnote{Indeed, if $R$ is an $F$-finite ring and $S$ is a direct limit of \etale{} maps (\ie an ind-\etale{} map), then $R^{1/p} \otimes_R S = S^{1/p}$ so that $S \subseteq S^{1/p}$ is finite.}.
In other words, we see that $s(\O_{X,x}^{\sh}) \geq \delta$ and $\O_{X,x}^{\sh}$ is $F$-finite for every geometric point $x$ of $X$.
By \cite[Theorem A]{CarvajalSchwedeTuckerEtaleFundFsignature}, we thus know that
\[
\big|\pi_1^{\et}(\Spec \O_{X,x}^{\sh} \setminus Z^{\sh})\big| \leq 1/\delta
\]
for each geometric point $x$ of $X$, where $Z$ is the singular locus of $X$.

Construct the stratification $\{ Z_i \}_{i \in I}$ as in \autoref{prop.GabberStratification}, where $Z = X_{\sing}$; assume the $Z_i$ are irreducible.  Observe that the $Z_i \cap X_k$ also stratify the open sets $X_k \subseteq X$ (although we may lose some pieces of the stratification).
Let $s_i$ be a geometric generic point of $Z_i$ and let $T_{s_i} := \Spec \O_{X, s_i}^{\sh} \setminus Z^{\sh}$ be its ``tubular neighborhood''. Pull back the whole sequence $X \leftarrow Y_1 \xleftarrow{\gamma_1} Y_2 \xleftarrow{\gamma_2} \cdots $ to $T_{s_i}$ to obtain
\[ T_{s_i} \xleftarrow{\;\;\gamma_{0, i}\;\;} T_{1, s_i} \xleftarrow{\;\;\gamma_{1, i}\;\;} T_{2, s_i} \xleftarrow{\;\;\gamma_{2, i}\;\;} \cdots. \]
Then all but finitely many of the $\gamma_{k, i}$ become trivial, as $\pi_1^{\et}(T_{s_i})$ is finite (of course, it is possible that $T_{k, s_i}$ becomes empty for $k \gg 0$ if $s_i \notin X_k$).
Note here is where we use the hypothesis that each $\eta_k : Y_k \to X_k \subseteq X$ is generically Galois, in particular Galois over the regular locus $U_k$ of $X_k$.

Now, given that $I$ is finite, one can pick $N \gg 0$ such that $\gamma_{n, s_i}$ is trivial for all $i \in I$ and all $n \geq N$. 
But by \autoref{rem.StratificationBehavesWellUnderPullback}, the inverse images of the generic points $s_i$ of the $Z_i$ in $Y_n$ are exactly the generic points $s_{n,j}$ of a stratification $\{ Z_{n,j} \}_{j \in J_n}$ of $Y_n$ satisfying the conclusion of \autoref{prop.GabberStratification}, with the $Z_{n,j}$ irreducible. Since
\[
\xymatrix{
\gamma_n\colon Y_{n+1} \ar@{->>}[r]^-{\text{finite}} & \eta_{n}^{-1}(X_{n+1}) \ar@{^{(}->}[r]^-{\text{open}} & Y_{n}
}
\]
is trivial after base changing with
\[ T_{s_{n,j}} := \Spec \O_{Y_n, s_{n,j}}^{\sh} \setminus \text{preimage of $\eta_n^{-1}(Z)$} \]
for $n \ge N$ and $j \in J_n$, by \autoref{clm.EtaleIffLocallyTrivial} we see that $\gamma_n$ is \'etale over the geometric generic point of every stratum $Z_{n,j}$ (of course, it might miss some completely).
Now applying \autoref{prop.StrataAndCovers} shows that $\gamma_n$ is finite \'etale over every point of $\eta_n^{-1}(X_{n+1})$, which completes the proof.
\end{proof}

The following corollary follows immediately.

\begin{corollary}\textnormal{(\cf~\cite[Theorem 1.5]{GrebKebekusPeternellEtaleFundamental})} \label{cor.MaximalQuasiEtale}
Suppose that $X$ is an $F$-finite Noetherian integral strongly $F$-regular scheme.
Then there exists a finite quasi-\'etale generically Galois cover $\rho\colon \tld X \to X$ with $\tld X$ normal and which satisfies the following property:

Every finite \'etale cover of the regular locus of $\tld X$ extends to a finite \'etale cover of $\tld X$.
Equivalently, the map $\piet(\tld X_{\reg}) \to \piet(\tld X)$ induced by the inclusion of the regular locus is an isomorphism.
\end{corollary}

We give two proofs of this result.
The first one uses \autoref{thm.QuasiEtaleStabilize}, while the second one is a direct proof which emphasizes the Galois correspondence.

\begin{proof}[First proof of \autoref{cor.MaximalQuasiEtale}]
Suppose not, then for every finite quasi-\'etale generically Galois cover $Y \to X$ there exists a further finite quasi-\'etale cover $Y' \to Y$, which we may also assume to be generically Galois over $X$, such that $Y' \to Y$ is not \'etale.
Repeating this process now with $Y'$ taking the role of $Y$, we obtain a sequence of covers contradicting the conclusion of \autoref{thm.QuasiEtaleStabilize}.

The proof of the statement about the map $\piet(\tld X_{\reg}) \to \piet(\tld X)$ is the same as~\cite[Step 2 of the proof of Thm.~1.5]{GrebKebekusPeternellEtaleFundamental}, and is thus omitted.
\end{proof}

\begin{proof} [Second proof of \autoref{cor.MaximalQuasiEtale}]
Consider the stratification $\{ Z_i \}_{i \in I}$ as in \autoref{prop.GabberStratification} where each $Z_i$ is connected.  Set $Z = X_{\sing}$.
As before, for every $Z_i$ choose $s_i$ to be a geometric generic point.
As observed in the proof of \autoref{prop.StrataAndCovers}, if $t$ is a geometric point of $Z_i$ generizing to $s_i$ (as in \autoref{prop.GabberStratification}), then $\pi_1^{\et}(T_t)$ and $\pi_1^{\et}(T_s)$ have a common image in $\pi_1^{\et}(U)$, where $U = X \setminus Z$. 
Let then $G_i \subseteq \pi_1^{\et}(U)$ denote this common image.
Note this common image is only unique up to conjugation.
In particular, since the index set $I$ is finite, there are only finitely many $G_i$ up to conjugation.

\begin{claim}
There is an open normal subgroup $H \leq \pi_1^{\et}(U)$ that intersects $G_i$ trivially for every $i$.
\end{claim}

\begin{proof}[Proof of claim]
Note that $\pi_1^{\et}(U)$ is a profinite group
\[ \pi_1^{\et}(U) = \varprojlim Q_j \]
where the $Q_j$ are quotients of $\pi_1^{\et}(U)$.
Thus each $G_i$ maps injectively to $Q_j$ for large enough $j \gg 0$.
Choose a large enough $j$ that works for all $G_i$ and let $H = \ker(\pi_1^{\et}(U) \to Q_j)$.
In particular, $H$ is an open normal subgroup that intersects each $G_i$ trivially and hence by normality, it also intersects any conjugate trivially.
This proves the claim.
\end{proof}

Returning to the proof, by the Galois correspondence, there is a Galois cover $\varrho\colon \tld U \to U$ such that $\Aut_U(\tld U) = Q_j$ and $\pi_1^{\et}\bigl(\tld U\bigr) = H \subseteq \piet(U)$. 
Let $\rho\colon \tld X \to X$ be the integral closure of $X$ in $K(\tld U)$, which is finite and its pullback to $U$ is exactly $\varrho$.
In particular, $\rho$ is quasi-\'etale and generically Galois.
In what remains, we prove that $\rho$ has the desired property. 

Let $V\to \tld U$ be a finite \'etale cover in $\textsf{F\'{E}t}\bigl(\tld U\bigr)$, we will extend it across $\tld X$. 
Take $\sigma\colon Y \to \tld X$ to be the integral closure of $\tld X$ in $K(V)$. 
It suffices to prove that $Y \to \tld X$ is \'etale.

Let $W \subseteq \tld X$ be the complement of the branch locus of $\sigma$. Let $\tld Z$ denote the inverse image of $Z$ in $\tld X$ and note that $\tld U$ is the inverse image of $U$ in $\tld X$, so that $\tld U = \tld X \setminus \tld Z$.
Now $\tld U \subseteq W$ and we want to show that $W = \tld X$. 
By \autoref{clm.EtaleIffLocallyTrivial}, a geometric point $x \in \tld X$ belongs to $W$ if and only if the pullback of $\sigma\colon Y \to \tld X$ to $\tld T_x := \Spec \O_{\tld X, x}^{\sh} \setminus \tld{Z}^{\sh}$ is trivial, where $\tld Z^{\sh}$ is the preimage of $\tld Z$.
For any geometric point $x \in \tld X$, $\tld T_x \to \tld X$ factors through $\tld U$.
Thus the pullback of $\sigma\colon Y \to \tld X$ to $\tld T_x$ coincides with the pullback of $V \to \tld U$ to $\tld T_x$.
We see that $x \in W$ if and only if the pullback of $V \to \tld U$ to $\tld T_x$ is trivial.
Hence it suffices to show that for all geometric points $x \in \tld X$ and all $V / \tld U \in \textsf{F\'{E}t}(\tld U)$, the pullback of $V \to \tld U$ to $\tld T_x$ is trivial.
Equivalently, we want the induced homomorphism of fundamental groups $\piet \bigl( \tld T_x \bigr) \to \piet \bigl( \tld U \bigr)$ to be zero~\cite[5.2.3]{MurreLecturesFundamentalGroups}. 
This is argued below.

Let $t \in X$ be the image of the geometric point $x \in \tld X$. We have a commutative diagram
\[
\xymatrix{
\tld T_x \ar[r] \ar[d] & \tld U \ar[d] \\
T_t \ar[r] & U
}
\]
with finite \'etale vertical morphisms.
By applying $\piet(-)$ we obtain the commutative square
\[
\xymatrix{
\pi_1^{\et}\bigl(\tld T_x\bigr) \ar@{^{(}->}[d] \ar[r] & H \ar@{^{(}->}[d] \\
\pi_1^{\et}(T_t) \ar[r] & \pi_1^{\et}(U)
}
\]
However, by construction $H$ meets the image of $\pi_1^{\et}(T_t)$ in $\pi_1^{\et}(U)$ trivially, which forces the top map in this square to be zero, as required.
\end{proof}


If $X$ is a strongly $F$-regular proper variety over an $F$-finite field $k$, and $U = X_{\reg}$, then it follows from \cite[Theorem 7.6]{SchwedeTuckerTestIdealFiniteMaps} that any \etale{} cover $U$ is cohomologically tame and hence tame in all the senses of \cite{KerzSchmidtOnDifferentNotionsOfTameness}.
From this we immediately obtain the following corollary.

\begin{corollary} \label{cor.MaximalTameCover}
If $X$ is a strongly $F$-regular proper variety over an $F$-finite field $k$ with $U = X_{\reg}$, then the map $\tld X \to X$ from \autoref{cor.MaximalQuasiEtale} restricts to a tame \etale{} cover of $U$.
\end{corollary}

We also obtain a local version of \autoref{cor.MaximalQuasiEtale}.

\begin{corollary} \textnormal{(cf.~\cite[Theorem 1.9]{GrebKebekusPeternellEtaleFundamental})} \label{cor.LocalMaximalQuasiEtale}
Suppose that $X$ is an $F$-finite Noetherian integral strongly $F$-regular scheme and that $x \in X$ is a point.
Then there exists a Zariski-open neighborhood $x \in X^\circ \subseteq X$ and a finite quasi-\'etale generically Galois cover $\rho\colon \tld X^\circ \to X^\circ$ with $\tld X^\circ$ normal and which satisfies the following property:

For every further Zariski-open neighborhood $x \in W \subseteq X^\circ$, with $\tld W = \rho^{-1}(W)$, any further finite quasi-\'etale cover $\overline{W} \to \tld W$ is \'etale.
Equivalently, $\piet(\tld W_{\reg}) \to \piet(\tld W)$ is an isomorphism.
\end{corollary}

\begin{proof}
Suppose not, \ie~assume that for every $x \in X^\circ \subseteq X$ and every finite quasi-\'etale generically Galois $\tld{X^\circ} \xrightarrow\rho X^\circ$, there is $x \in W \subseteq X^\circ$ and a finite quasi-\'etale $\overline W \to \rho^{-1}(W)$ that is not \'etale.
By taking Galois closure, we may assume that $\overline W \to W$ is Galois.

Apply this assumption with $X^\circ := X$ and $\rho := \id_X$.
We obtain a map $\overline W \to W$, which we denote $\eta_1\colon Y_1 \to X_1$.
Applying the assumption again, this time to $\eta_1$, we get a map $\eta_2\colon Y_2 \to X_2$ together with maps $\gamma_1\colon Y_2 \to Y_1$ and $X_2 \hookrightarrow X_1$.
Inductively, we construct a diagram as in \autoref{thm.QuasiEtaleStabilize}, but where none of the $\gamma_i$ are \'etale.
This contradicts \autoref{thm.QuasiEtaleStabilize}.

\end{proof}

\begin{remark}
Note that one cannot always take $X^{\circ} = X$ in the statement of \autoref{cor.LocalMaximalQuasiEtale}, even in characteristic zero.
Indeed let $X$ be the \emph{projective} quadric cone $V(x^2 - yz) \subseteq \bP^3_{\bC}$.
Then $X_{\reg}$ is simply connected (it is an $\mathbb A^1$-bundle over $\bP^1_{\bC}$), so every finite quasi-\'etale cover of $X$ is trivial.
On the other hand, the \emph{affine} quadric cone $U \subseteq X$ does have a finite quasi-\'etale cover that is not \'etale, corresponding to $k[x^2, xy, y^2] \subseteq k[x,y]$.

In characteristic 2, $\pi_1^{\et}$ of the punctured quadric cone singularity is trivial and hence this computation does not work (the local cover above is in fact inseparable).
However, one can obtain the same conclusion over an algebraically closed field of characteristic $p > 2$.
Let $X' \to X$ be a finite quasi-\'etale generically Galois cover of $X$ that is not \'etale.
Let $O \in X$ be the cone point.
It follows that
\[ X' \times_X \Spec \O_{X, O}^{\sh} \]
is a disjoint union of copies of $V \to \O_{X,O}^{\sh}$ where $V$ is the regular 2-to-1 cover of $\Spec \O_{X, O}^{\sh}$ (corresponding to $k[x^2, xy, y^2] \subseteq k[x,y]$).

Let $L$ be a ruling of the cone and consider $X' \times_X L \to L$.  This map is \etale{} except over the cone point $O \in L \subseteq X$, and over that point $X' \times_X L$ is not even reduced.  However, if one takes $L' = (X' \times_X L)_{\red}$ then the computation $k\llbracket x^2, xy, y^2 \rrbracket/\langle x^2, xy \rangle \subseteq k\llbracket x,y \rrbracket /\langle x^2, xy \rangle$ shows that $L'$ is normal and furthermore that $L' \to L$ is ramified of order 2 over $O$ on each connected component.  In particular, $L' \to L$ is ramified over a single point and that ramification is tame (since $p > 2$), but $L \cong \bP^1$ and so this is a contradiction.
Hence any quasi-\etale{} cover of $X$ is in fact \etale{} and so we cannot take $X^{\circ} = X$ in \autoref{cor.LocalMaximalQuasiEtale} just as in characteristic zero.
\end{remark}

\begin{remark}
We believe one can prove \autoref{cor.LocalMaximalQuasiEtale} using a strategy similar to \autoref{cor.MaximalQuasiEtale}.  In particular, we can first assume that $x$ is in the closure of any stratum $Z_i$ (if not, shrink $X$ to remove those strata).  Then use
\[
\pi_1^{\et}\big((\Spec \O_{X,x}) \setminus \text{(inverse image of $Z$)}\big)
\]
as the replacement for $\pi_1^{\et}(U)$ in the proof of \autoref{cor.LocalMaximalQuasiEtale}.  The $H$ we obtain produces a cover that is quasi-\'etale over a neighborhood $X^\circ$ of $x$ and which satisfies the desired property.
\end{remark}

Just as in \cite[Theorem 1.10]{GrebKebekusPeternellEtaleFundamental}, we obtain a result on simultaneous index-one covers.

\begin{corollary} \label{cor.SimultaneousIndex1Cover}
Suppose that $x \in X$ and $\tld X^\circ \twoheadrightarrow X^\circ \subseteq X$ is as in \autoref{cor.LocalMaximalQuasiEtale}.
Suppose further that $X^\circ$ is chosen to be affine (or quasi-projective over an affine scheme).
\begin{itemize}
\item[(a)]  If $\tld D$ is any $\bZ_{(p)}$-Cartier\footnote{This just means it is $\bQ$-Cartier with index not divisible by $p$.} divisor on $\tld X^\circ$, then $\tld D$ is Cartier.
\item[(b)]  There exists an integer $N > 0$ so that if $D$ is any $\bZ_{(p)}$-Cartier divisor on $X^\circ$, then $N \cdot D$ is Cartier.
\end{itemize}
\end{corollary}

\begin{proof}
The proof is exactly the same as the proof of \cite[Theorem 1.10]{GrebKebekusPeternellEtaleFundamental}, and so we omit it.
Of course, the main idea is to take a cyclic cover.
The hypothesis that $X^\circ$ is affine (or more generally quasi-projective over an affine) is assumed so that for any finite collection of points $S \subseteq X^\circ$ and for any line bundle $\sL$ on $X$, that there exists an open neighborhood of $S$ which trivializes $\sL$.
\end{proof}

Just as in \cite{GrebKebekusPeternellEtaleFundamental}, \autoref{cor.SimultaneousIndex1Cover} implies that for any $F$-finite Noetherian integral strongly $F$-regular scheme, there exists an integer $N > 0$ such that if $D$ is $\bZ_{(p)}$-Cartier, then $N \cdot D$ is Cartier.
The point is that by quasi-compactness, $X$ can be covered by finitely many open $X^\circ$ satisfying \autoref{cor.SimultaneousIndex1Cover}.


\bibliographystyle{skalpha}
\bibliography{MainBib}

\def\cfudot#1{\ifmmode\setbox7\hbox{$\accent"5E#1$}\else
  \setbox7\hbox{\accent"5E#1}\penalty 10000\relax\fi\raise 1\ht7
  \hbox{\raise.1ex\hbox to 1\wd7{\hss.\hss}}\penalty 10000 \hskip-1\wd7\penalty
  10000\box7}
\providecommand{\bysame}{\leavevmode\hbox to3em{\hrulefill}\thinspace}
\providecommand{\MR}{\relax\ifhmode\unskip\space\fi MR}
\providecommand{\MRhref}[2]{%
  \href{http://www.ams.org/mathscinet-getitem?mr=#1}{#2}
}
\providecommand{\href}[2]{#2}
\begin{thebibliography}{DSPY16}

\bibitem[Tra14]{TravauxdeGabber}
\emph{Travaux de {G}abber sur l'uniformisation locale et la cohomologie \'etale
  des sch\'emas quasi-excellents}, S{\'e}minaire {\`a} l'{\'E}cole
  Polytechnique 2006--2008. [Seminar of the Polytechnic School 2006--2008],
  With the collaboration of Fr{\'e}d{\'e}ric D{\'e}glise, Alban Moreau, Vincent
  Pilloni, Michel Raynaud, Jo{\"e}l Riou, Beno{\^{\i}}t Stroh, Michael Temkin
  and Weizhe Zheng, Ast{\'e}risque No. 363-364 (2014) (2014), 2014.
  {\sf\scriptsize 3309086}

\bibitem[CST16]{CarvajalSchwedeTuckerEtaleFundFsignature}
{\sc J.~{Carvajal-Rojas}, K.~Schwede, and K.~Tucker}: \emph{Fundamental groups
  of {$F$}-regular singularities via {$F$}-signature}, arXiv:1606.04088.

\bibitem[DSPY16]{deStefaniPolstraYao-Globalizing}
{\sc A.~De~Stefani, T.~Polstra, and Y.~Yao}: \emph{Globalizing {F}-invariants},
  arXiv:1608.08580.

\bibitem[GW10]{GortzWedhornAlgebraicGeoemtryI}
{\sc U.~G\"ortz and T.~Wedhorn}: \emph{{Algebraic geometry I. Schemes. With
  examples and exercises.}}, {Advanced Lectures in Mathematics. Wiesbaden:
  Vieweg+Teubner. vii, 615~p. EUR~59.95 }, 2010 (English).

\bibitem[GKP13]{GrebKebekusPeternellEtaleFundamental}
{\sc D.~{Greb}, S.~{Kebekus}, and T.~{Peternell}}: \emph{{\'{E}tale fundamental
  groups of Kawamata log terminal spaces, flat sheaves, and quotients of
  Abelian varieties}}, To appear in Duke Mathematical Journal, arXiv:1307.5718.

\bibitem[HW02]{HaraWatanabeFRegFPure}
{\sc N.~Hara and K.-I. Watanabe}: \emph{F-regular and {F}-pure rings vs. log
  terminal and log canonical singularities}, J. Algebraic Geom. \textbf{11}
  (2002), no.~2, 363--392. {\sf\scriptsize MR1874118 (2002k:13009)}

\bibitem[Har94]{HartshorneGeneralizedDivisorsOnGorensteinSchemes}
{\sc R.~Hartshorne}: \emph{Generalized divisors on {G}orenstein schemes},
  Proceedings of Conference on Algebraic Geometry and Ring Theory in honor of
  Michael Artin, Part III (Antwerp, 1992), vol.~8, 1994, pp.~287--339.
  {\sf\scriptsize MR1291023 (95k:14008)}

\bibitem[KS10]{KerzSchmidtOnDifferentNotionsOfTameness}
{\sc M.~Kerz and A.~Schmidt}: \emph{On different notions of tameness in
  arithmetic geometry}, Math. Ann. \textbf{346} (2010), no.~3, 641--668.
  {\sf\scriptsize MR2578565}

\bibitem[Mil80]{MilneEtaleCohomology}
{\sc J.~S. Milne}: \emph{\'{E}tale cohomology}, Princeton Mathematical Series,
  vol.~33, Princeton University Press, Princeton, N.J., 1980. {\sf\scriptsize
  559531}

\bibitem[Mil13]{MilneLecturesOnEtaleCohomology}
{\sc J.~S. Milne}: \emph{Lectures on \'{E}tale cohomology}, 2013,
  http://www.jmilne.org/math/CourseNotes/LEC.pdf.

\bibitem[Mur67]{MurreLecturesFundamentalGroups}
{\sc J.~P. Murre}: \emph{Lectures on an introduction to {G}rothendieck's theory
  of the fundamental group}, Tata Institute of Fundamental Research, Bombay,
  1967, Notes by S. Anantharaman, Tata Institute of Fundamental Research
  Lectures on Mathematics, No 40. {\sf\scriptsize 0302650}

\bibitem[Pol15]{PolstraSemicontinuityFsignature}
{\sc T.~Polstra}: \emph{Uniform bounds in {$F$}-finite rings and lower
  semi-continuity of the {$F$}-signature}, arXiv:1506.01073.

\bibitem[PT16]{TuckerPolstraUniformApproach}
{\sc T.~Polstra and K.~Tucker}: \emph{{$F$}-signature and hilbert-kunz
  multipicity: a combined approach and comparison}, arXiv:1608.02678.

\bibitem[ST14]{SchwedeTuckerTestIdealFiniteMaps}
{\sc K.~Schwede and K.~Tucker}: \emph{On the behavior of test ideals under
  finite morphisms}, J. Algebraic Geom. \textbf{23} (2014), no.~3, 399--443.
  {\sf\scriptsize 3205587}

\bibitem[Smi16]{SmirnovUpperSemicontinuity}
{\sc I.~Smirnov}: \emph{Upper semi-continuity of the {H}ilbert-{K}unz
  multiplicity}, Compos. Math. \textbf{152} (2016), no.~3, 477--488.
  {\sf\scriptsize 3477638}

\bibitem[{Sta}16]{stacks-project}
{\sc T.~{Stacks Project Authors}}: \emph{\itshape stacks project}, 2016.

\bibitem[Xu14]{XuFinitenessOfFundGroups}
{\sc C.~Xu}: \emph{Finiteness of algebraic fundamental groups}, Compos. Math.
  \textbf{150} (2014), no.~3, 409--414. {\sf\scriptsize 3187625}

\bibitem[Yao06]{YaoObservationsAboutTheFSignature}
{\sc Y.~Yao}: \emph{Observations on the {$F$}-signature of local rings of
  characteristic {$p$}}, J. Algebra \textbf{299} (2006), no.~1, 198--218.
  {\sf\scriptsize MR2225772 (2007k:13007)}

\end{thebibliography}

\end{document}